\documentclass{amsart}
\usepackage{amssymb}
\usepackage{amsmath}
\usepackage{hyperref}

\newtheorem{theorem}{Theorem}[section]
\newtheorem{lemma}[theorem]{Lemma}

\newtheorem{corollary}[theorem]{Corollary}

\newtheorem{conjecture}[theorem]{Conjecture}

\theoremstyle{definition}

\DeclareMathOperator{\ord}{ord}
\DeclareMathOperator{\supp}{supp}

\begin{document}
\title{Direct zero-sum problems for certain groups of rank three}
\author{Benjamin Girard} 
\address{Sorbonne Universit\'e, Universit\'e Paris Diderot, CNRS, Institut de Math\'ematiques de Jussieu - Paris Rive Gauche, IMJ-PRG, F-75005, Paris, France}
\email{\texttt{benjamin.girard@imj-prg.fr}}
\author{Wolfgang A.~Schmid}
\address{Universit\'e Paris 13, Sorbonne Paris Cit\'e, LAGA, CNRS, UMR 7539, Universit\'e Paris 8, 
F-93430, Villetaneuse, France, and 
Laboratoire Analyse, G\'eom\'etrie et Applications (LAGA, UMR 7539), COMUE Universit\'e Paris Lumi\`eres, Universit\'e Paris 8, 
CNRS, 93526 Saint-Denis cedex, France}
\email{\texttt{schmid@math.univ-paris13.fr}}


\begin{abstract}
We determine the exact value of the $\eta$-constant and the multiwise Davenport constants  for finite abelian groups of rank three having the form $G \simeq C_2 \oplus C_{n_2} \oplus C_{n_3}$ with $2 \mid n_2 \mid n_3$. Moreover,  we determine the Erd\H{o}s--Ginzburg--Ziv constant of these groups under the assumption that  $n_2/2$  has Property D or $n_2  = n_3$.   
\end{abstract}

\maketitle 

\section{Introduction}

A well-known direct zero-sum problem is to determine the Davenport constant of finite abelian groups. 
For such a group $(G,+,0)$, this constant, denoted by $\mathsf{D}(G)$, is defined as the smallest non-negative integer $t$ such that every sequence of $t$ elements from $G$ contains a non-empty subsequence whose terms sum to $0$. 

A closely related problem is to determine the Erd\H{o}s--Ginzburg--Ziv constant, denoted by $\mathsf{s}(G)$, which is defined in the same way except that one requires the existence of a subsequence whose sum is $0$ and whose length is equal to the exponent of the group.
A variant of this constant is $\eta(G)$, where one seeks a non-empty subsequence with sum $0$ whose length is at most the exponent of the group.   

The investigation of these zero-sum constants has been a topic of active research for more than fifty years.
We refer to \cite{GaoGero06,GeroRuzsa09,GeroKoch06,GrynkiewiczBOOK}  for detailed expositions.  
Some results are also recalled in the next section. 
The exact values of these three constants are known for every finite abelian group of rank at most two, and only for fairly special types of groups of higher rank.
Even for groups of rank three, 
that is, $G \simeq C_{n_1} \oplus C_{n_2} \oplus C_{n_3}$ with $1< n_1 \mid n_2 \mid n_3$,
the problem of determining these constants is wide open. 
For example, when $n_1=2,$ the Davenport constant is known, but the exact values of the other two constants defined above is not.

In the present paper, we obtain these values for the $\eta$-constant and, assuming a now well-supported conjecture, for the Erd\H{o}s--Ginzburg--Ziv constant as well. 
Our results confirm Gao's conjecture (Conjecture \ref{GaoConjecture}) for this type of groups, and generalize previous results obtained in the case $n_1 =n_2 = 2$  (see \cite[Theorem 1.2(1)]{FanZhong16} and \cite[Theorem 1.3]{FanGaoPengWangZhong13}). Moreover, they show that recent results of Luo \cite{Luo17} are essentially optimal.  
In addition, we determine the multiwise Davenport constants for this type of groups (see the subsequent section for the definition).  
For a more detailed overview of our results and how they relate to the existing literature on the subject, we refer to Section \ref{sec_nr}. 

\section{Preliminaries}
\label{sec_prel}

We recall some notation and results; for more detailed information we refer again to \cite{GaoGero06,GeroRuzsa09,GeroKoch06,GrynkiewiczBOOK}. All intervals in this paper are intervals of integers, specifically $[a,b] = \{z \in \mathbb{Z} \colon a \le z \le b \}$.

Let $G$ be a finite abelian group, written additively. For each $g$ in $G$, we denote by $\ord(g)$ its order in $G$. 
For a subset $A \subseteq G$ we denote by  $\left\langle A \right\rangle$ the subgroup it generates; we say that $A$ is a \emph{generating set} 
if  $\left\langle A \right\rangle = G$. We say that elements $g_1 , \dots, g_k$ are \emph{independent} if 
$\sum_{i=1}^k a_i g_i= 0$, with integers $a_i$, implies that $a_ig_i = 0$ for each $i$; we say that a set is independent when its elements
are independent.  

By $\exp(G)$ we denote the exponent of $G$, that is the least common multiple of the orders of elements of $G$.  By 
$\mathsf{r}(G)$ we denote the rank of $G$, that is the minimum cardinality of a generating subset of $G$. 
For $n$ a positive integer we denote by $C_n$ a cyclic group of order $n$. 

For a finite abelian group $G$ there exist uniquely determined integers $1 < n_1 \mid \dots \mid n_r$ such that $G \simeq C_{n_1} \oplus \dots \oplus C_{n_r}$. For $|G|>1$ we have $\mathsf{r}(G)= r$ and $\exp(G)=n_r$; the rank of a group of cardinality $1$ is $0$ and its exponent is $1$.
  
\medskip
By a \emph{sequence} over $G$ we mean an element of the free abelian monoid over $G$. In other words, this is a finite sequence of $\ell$ elements from $G$, where repetitions are allowed and the order of elements is disregarded. We use multiplicative notation for sequences. We denote its neutral element, that is the sequence  of length zero, simply by $1$. Let
\[
S = g_1  \cdots  g_{\ell} = \displaystyle\prod_{g \in G} g^{\mathsf{v}_g(S)}
\]
be a sequence over $G$, where, for all $g \in G$, $\mathsf{v}_g(S)$ is a non-negative integer called the \emph{multiplicity} of $g$ in $S$. Moreover $\ell$ is the  \emph{length} of $S$. 

A sequence $T$ over $G$ is said to be a \emph{subsequence} of $S$ if it is a divisor of $S$ in the free abelian monoid over $G$, that is if $\mathsf{v}_g(T) \le \mathsf{v}_g(S)$ for all $g \in G$; in this case we write $T \mid S$. For a subsequence $T$ of $S$ we set
\(ST^{-1} = \prod_{g \in G} g^{\left(\mathsf{v}_g(S)-\mathsf{v}_g(T)\right)}\), that is, it is the subsequence of $S$ such that 
$T (ST^{-1}) = S$. 

For subsequences $T_1, \dots, T_k$  of $S$ we say that they are \emph{disjoint subsequences} if 
$T_1 \cdots T_k$ is also a subsequence of $S$. Moreover, for sequences 
  $S_1, \dots, S_k$ over $G$ we denote by 
\[
\gcd(S_1, \dots, S_k)=\displaystyle\prod_{g \in G} g^{\min\left\{\mathsf{v}_g(S_i)\colon 1 \le i \le   k \right\} }
\]
the greatest common divisor of the sequences $S_1, \dots, S_k$ in the free abelian monoid. To avoid confusion we stress that it is not necessary for disjoint subsequences to have a trivial greatest common divisor.

\medskip
We call the set $\supp(S)=\{g \in G \mid \mathsf{v}_g(S)>0\}$ the \emph{support} of $S$, and $\sigma(S)=\sum^{\ell}_{i=1}g_i=\sum_{g \in G}\mathsf{v}_g(S)g$ the \emph{sum} of $S$. In addition, we say that $s \in G$ is a \emph{subsum} of $S$ if
\[
s=\displaystyle\sum_{i \in I} g_i \text{ for some }\emptyset \varsubsetneq I \subseteq [1,\ell].
\]

\noindent If $0$ is not a subsum of $S$, we say that $S$ is a \emph{zero-sumfree sequence}. If $\sigma(S)=0,$ then $S$ is said to be a \emph{zero-sum sequence}. 
If, moreover, one has $\sigma(T) \neq 0$ for all proper and non-empty subsequences $T \mid S$, then $S$ is called a \emph{minimal zero-sum sequence}.

We set
\[
\Sigma(S) = \{\sigma(T) \colon 1 \neq T \mid S\}.
\]
\noindent For every integer $k$, we also set  
\[
\Sigma_k(S)=\{\sigma(T) \colon T \mid S,\, |T| = k\}\] 
as well as 
\[ \Sigma_{\le k}(S)=\bigcup^k_{i=1} \Sigma_i(S) = \{\sigma(T) \colon 1 \neq T \mid S,\, |T| \le k\}.
\]

\medskip 
We now recall in more detail the definitions and results alluded to in the introduction. 

By $\mathsf{D}(G)$ we denote the smallest non-negative integer $t $ such that every sequence $S$ over $G$ of length $|S| \ge t$ contains a non-empty zero-sum subsequence. This number $\mathsf{D}(G)$ is called the \emph{Davenport constant} of the group $G$.
More generally, given an integer $k \ge 1$, we denote by $\mathsf{D}_k(G)$ the smallest non-negative integer $t $ such that every sequence $S$ over $G$ of length $|S| \ge t$ contains at least $k$ non-empty disjoint zero-sum subsequences. 

\medskip
Note that, by definition, $\mathsf{D}_1(G) = \mathsf{D}(G)$ for every finite abelian group $G$.
It is known that for every finite abelian group the sequence  $(\mathsf{D}_k(G))_{k \ge 1}$
 is eventually an arithmetic progression. 
 More precisely, one has the following result (see \cite[Lemma 5.1]{FreezeSchmid10}).  

\begin{theorem} 
Let $G$ be a finite abelian group. There exist $\mathsf{D}_0(G) \in \mathbb{N}$ and an integer $k_{0}\ge 1$ such that 
\[
\mathsf{D}_k(G) = \mathsf{D}_0(G) + k \exp(G), \text{ for each } k \ge k_0.
\]
\end{theorem}
Let $k_{\mathsf{D}}(G)$ denote the smallest possible value of $k_0$ in the above theorem.

\medskip
By $\eta(G)$ we denote the smallest non-negative integer $t$ such that every sequence $S$ over $G$ of length $|S| \ge t$ contains a non-empty zero-sum subsequence $S' \mid S$ of length $|S'| \le \exp(G)$. Such a subsequence is called a \emph{short zero-sum subsequence}.

\medskip
By $\mathsf{s}(G)$ we denote the smallest non-negative integer $t $ such that every sequence $S$ over $G$ of length $|S| \ge t$ contains a zero-sum subsequence $S' \mid S$ of length $|S'| = \exp(G)$. The number $\mathsf{s}(G)$ is called the \emph{Erd\H{o}s--Ginzburg--Ziv constant} of the group $G$. 

It is not hard to see that  $\mathsf{s}(G) \ge \eta(G) + \exp(G) - 1$  holds for each finite abelian group $G$.
It was conjectured by Gao that in fact equality always holds (see \cite[Conjecture 6.5]{GaoGero06}). 

\begin{conjecture}[Gao] \label{GaoConjecture} For every finite abelian group $G$, one has
\[
\mathsf{s}(G) = \eta(G) +  \exp(G)-1.
\]
\end{conjecture}

We now recall the values of $\eta(G)$ and $\mathsf{s}(G)$ as well as the ones of the multiwise Davenport constants for groups of rank at most two, see \cite[Theorem 5.8.3]{GeroKoch06} and \cite[Theorem 6.1.5]{GeroKoch06}. We parametrize these groups as  $C_m \oplus C_{mn}$ with 
$m, n \ge 1$ integers rather than $C_{n_1} \oplus C_{n_2}$ with integers $n_1 \mid n_2$ since later on this will be more convenient. 

\begin{theorem}\label{ranktwo}
Let $m, n \ge 1$ be two integers. Then
\[
\eta(C_m \oplus C_{mn})=2m+mn-2 \quad \text{ and } \quad \mathsf{s}(C_m \oplus C_{mn})=2m+2mn-3.
\]
In addition, for every integer $k \ge 1$,
\[
\mathsf{D}_k(C_m \oplus C_{mn})=m+k(mn)-1.
\]
In particular, choosing $m=1$, we have $\eta(C_n)=n$ and $\mathsf{s}(C_n)=2n-1$ as well as $\mathsf{D}_k(C_n)=kn$ for all $k \ge 1$.
\end{theorem}

Theorem \ref{ranktwo} shows that Conjecture \ref{GaoConjecture} is true for all finite abelian groups of rank at most two. 

\medskip
In the case of groups of rank at most two even the structure of extremal examples is well-understood. For cyclic groups in fact more is known, see, e.g., \cite{SavChen07,SavChen08}, yet we only recall what is needed in this paper.

A sequence $S$ over $C_n$ of length $n-1 = \eta(C_n)-1$ has no short zero-sum subsequence (and thus no non-empty zero-sum subsequence) if and only if $S=b^{n-1}$ for some generating element $b$ of $C_n$. 
A sequence $S$ over $C_n$ of length $2n-2= \mathsf{s}(C_n)-1$ has no zero-sum subsequence of length $n$ if and only if $S=c^{n-1}(c+b)^{n-1}$ for some $c \in C_n$ and some generating element $b$ of $C_n$. 

For the $\eta$-constant one has the following result. It was obtained in \cite{Wolfgang}; a result of Reiher \cite{reiherB} was crucial in the proof.  

\begin{theorem}
\label{Wolfgang_v} 
Let $H \simeq C_{m} \oplus C_{mn}$ with integers  $m\ge 2$ and $n \ge 1$. 
Every sequence $S$ over $H$ of length $\left|S\right|=\eta(H)-1$ not containing any short zero-sum subsequence has the following form:
\[S = b^{m-1}_1b^{sm-1}_2 \left(-xb_1 + b_2\right)^{(n+1-s)m-1}\]
where $\{b_1,b_2\}$ is a generating set of $H$ with $\ord(b_2) = mn$, $s \in [1,n]$, $x \in [1, m]$ with $\gcd(x,m)=1$ and either  
\begin{enumerate}
\item 
 $\{b_1,b_2\}$ is an independent generating set of $H$, or 
\item  $s= n$ and $x=1$.
\end{enumerate}
\end{theorem}

For the Erd\H{o}s--Ginzburg--Ziv constant a similar result is expected to hold true, yet it is so far only known conditionally or in special cases.  

A positive integer $m$ is said to have \emph{Property D} if every sequence $S$ over $C_m^2$ of length $|S|=\mathsf{s}(C_m^2)-1=4m-4$ and containing no zero-sum subsequence of length $m$ has the form $S=T^{m-1}$ for some sequence $T$ over $C_m^2$. 
We include the trivial case $m=1$ in our definition as it simplifies the statement of certain results. This property was introduced by Gao who made the following conjecture \cite[Conjecture 2]{Gao00}.

\begin{conjecture}[Gao]
Every positive integer has Property D.
\end{conjecture}

\medskip
For the time being, Property D has been proved to be multiplicative \cite[Theorem 1.4]{Gao00} in the sense that whenever $m,n$ have this property, then so does $mn$. 
Also, Property D is known to hold for  $p \in \{2,3,5,7\}$, hence for any $m$ of the form $m=2^{\alpha}3^{\beta}5^{\gamma}7^{\delta}$, where $\alpha,\beta,\gamma,\delta \ge 0$ are non-negative integers (see \cite[Theorem 1.5]{Gao00} and \cite[Theorem 3.1]{SuryThanga02}).

\medskip
Whenever an integer $m $ satisfies Property D, the sequences over $H \simeq C_m \oplus C_{mn}$ of length $\mathsf{s}(H)-1$ and not containing any zero-sum subsequence of length $\exp(H)$ can be fully characterized for every integer $n \ge 1$ (see \cite[Theorem 3.1(2)]{Wolfgang}).

\begin{theorem}
\label{Wolfgang bis_v} 
Let $H \simeq C_{m} \oplus C_{mn}$, where  $m \ge 2$ satisfies Property D and $n \ge 1$. 
Every sequence $S$ over $H$ of length $\left|S\right|=\mathsf{s}(H)-1$ not containing any zero-sum subsequence of length $\exp(H)$ has the following form:
\[S = c^{tm-1} (b_1+c)^{(n+1 -t)m-1}(b_2+c)^{sm-1} \left(-xb_1 + b_2+c \right)^{(n+1-s)m-1}\]
where $c\in H$, $\{b_1,b_2\}$ is a generating set of $H$ with $\ord(b_2) = mn$, $s,t \in [1,n]$,  $x \in [1, m]$ with $\gcd(x,m)=1$ and either   
\begin{enumerate}
\item
 $\{b_1,b_2\}$ is an independent generating set of $H$, or 
\item  $s=t= n$ and $x=1$.
\end{enumerate}
\end{theorem}

For definiteness we briefly recap some properties of the generating sets in the above result. Since $\{b_1,b_2\}$ is a generating set of $H$ and $\ord(b_2)=mn$, the equalities
\[
|H| = |\left\langle b_1 \right\rangle + \left\langle b_2 \right\rangle| = \frac{|\left\langle b_1 \right\rangle| |\left\langle b_2 \right\rangle|}{|\left\langle b_1 \right\rangle \cap \left\langle b_2 \right\rangle|}
\]
imply that $\ord(b_1)=md$ where $d=|\left\langle b_1 \right\rangle \cap \left\langle b_2 \right\rangle|$ is a positive divisor of $n$. 
More precisely, we have 
\[
\left\langle mb_1 \right\rangle = \left\langle b_1 \right\rangle \cap \left\langle b_2 \right\rangle = \left\langle m\frac{n}{d}b_2 \right\rangle.
\]
In particular, every element $h \in H$ can be written $h=a_1b_1+a_2b_2$ with $a_1\in [0,m-1]$ and $a_2 \in [0,mn-1]$.
In addition, it is easily seen that $\{b_1,b_2\}$ is an independent generating set of $H$ if and only if $\left\langle b_1 \right\rangle \cap \left\langle b_2 \right\rangle=\{0\}$, that is to say if and only if $d=1$.
Finally, whenever $d > 1$, there is a unique integer $\ell \in [1,d-1]$ relatively prime to $d$ such that $mb_1=\ell m(n/d)b_2$.

\medskip
We end by recalling the result on the Davenport constant for groups of the form $C_2 \oplus C_{2m} \oplus C_{2mn}$, which we mentioned in the introduction and that we need in the proof of our result on the multiwise Davenport constants.  

\begin{theorem}\label{D}
Let $m, n \ge 1$ be two integers. Then 
\[
\mathsf{D}(C_2 \oplus C_{2m} \oplus C_{2mn}) = 2m+2mn.
\]
\end{theorem}

The proof of the above result involved two parts. First, the claim was established conditionally on a result on the structure of the set of subsums of zero-sumfree sequences of maximal length over a group of rank two; this motivated the definition of the $\nu$-invariant (see, e.g., \cite[Definition 2.1]{GaoGero06}). Then, this property was established. The first part dates back to the very beginning of investigations of the Davenport constant (see \cite{EmdeBoas69}). The second part was only completed much later when Property B (and thus Property C) was established by Reiher \cite{reiherB} (see in particular Section 11). For further context, see for instance \cite{Gao_rank3} or \cite[Section 4.1]{Schmid11}.   

\section{New results}
\label{sec_nr}

As mentioned in the introduction we investigate zero-sum constants for groups of rank three of the form $C_2 \oplus C_{n_2} \oplus C_{n_3}$ where $2 \mid n_2 \mid n_3$. For ease of notation we will use a different parametrization, namely $C_2 \oplus C_{2m} \oplus C_{2mn}$ with $m,n \ge 1$. 

We determine $\eta(C_2 \oplus C_{2m} \oplus C_{2mn})$ for all $m,n \ge 1$, and $\mathsf{s}\left(C_2 \oplus C_{2m} \oplus C_{2mn}\right)$ under the condition that $n=1$ or $m$ has Property D. We recall that the constants were known for $m=1$, see  \cite[Theorem 1.2(1)]{FanZhong16} and \cite[Theorem 1.3]{FanGaoPengWangZhong13}; in this case even the inverse problem is solved \cite{GirardSchmid2}. 
Moreover, we determine $\mathsf{D}_k\left(C_2 \oplus C_{2m} \oplus C_{2mn}\right)$ for all $k,m,n \ge 1$; as recalled the case $k=1$ and the case $m=n=1$ were known.
We will see that there is a quite significant difference between the two cases $n=1$ and $n \neq 1$. 

Our approach to determining $\eta(C_2 \oplus C_{2m} \oplus C_{2mn})$ is similar to the one used for the Davenport constant, which we recalled above. In particular, the property of the set of restricted subsums established in Lemma \ref{lem_sums} resembles the property underlying the definition of the $\nu$-invariant. The subsequent result on $\mathsf{s}\left(C_2 \oplus C_{2m} \oplus C_{2mn}\right)$ is obtained by establishing Gao's conjecture for this group using a generalization of a well-known technique (see Lemma \ref{lem_exp-1}). The proof of the result for  $\mathsf{D}_k\left(C_2 \oplus C_{2m} \oplus C_{2mn}\right)$ also uses the result on  $\eta \left(C_2 \oplus C_{2m} \oplus C_{2mn}\right)$.      

\begin{theorem}\label{eta}
Let $m\ge 1$ and $n\ge 2$ be two integers. Then 
\[
\eta(C_2 \oplus C_{2m} \oplus C_{2m}) = 6m + 2
\]
and
\[
\eta(C_2 \oplus C_{2m} \oplus C_{2mn}) = 4m + 2mn.
\]
\end{theorem}

\medskip
When $G \simeq C_2\oplus C_{2m} \oplus C_{2m}$ and $m$ is a power of $2$, $G$ is a finite abelian $2$-group such that 
\[
\mathsf{D}(G) = 2\exp(G) \text{ and } 2\mathsf{D}(G)-\exp(G) < \eta(G),
\] 
thus showing that a recent result of Luo \cite[Theorem 1.6]{Luo17} is optimal in the sense that for $2\mathsf{D}(G)-\exp(G) = \eta(G)$ to hold, the condition $\mathsf{D}(G) \le 2\exp(G)-1$ in the statement of the theorem cannot be replaced by $\mathsf{D}(G) \le 2\exp(G)$.

\begin{theorem}\label{EGZ}
Let $m\ge 1$ and $n\ge 2$ be two integers. Then 
\[
\mathsf{s}(C_2 \oplus C_{2m} \oplus C_{2m}) = 8m+1.
\]
Moreover, if $m$ has Property D, then 
\[
\mathsf{s}(C_2 \oplus C_{2m} \oplus C_{2mn}) = 4m + 4mn -1.
\]
\end{theorem}

In combination, the two results imply that Gao's Conjecture \ref{GaoConjecture} holds true for these types of groups.  

\begin{corollary}
Let $m, n \ge 1$ be two integers. 
If $n=1$ or $m$ has Property D, then Conjecture \ref{GaoConjecture} holds true for $C_2 \oplus C_{2m} \oplus C_{2mn}$.
\end{corollary}

We end this section with our result on the multiwise Davenport constants. 

\begin{theorem}\label{D_k}
Let $G \simeq C_2 \oplus C_{2m} \oplus C_{2mn},$ where $m,n \ge 1$ are integers.
If $n \ge 2$, then $\mathsf{D}_0(G)=2m$ and $k_{\mathsf{D}}(G)=1$. 
If $n=1$, then $\mathsf{D}_0(G)=2m+1$ and $k_{\mathsf{D}}(G)=2$. 
\end{theorem}

Note that the case $n=1$ extends to all $m \ge 1$ a result of Delorme, Ordaz and Quiroz \cite[Lemma $3.7$]{DelOrdQui01} stating that $\mathsf{D}_0(C^3_2)=3$ and $k_{\mathsf{D}}(C^3_2)=2$.

\medskip
When $G \simeq C_2\oplus C_{2m} \oplus C_{2m}$, where $m \ge 1$, we have
\[
\mathsf{D}(G) = \mathsf{D}(C_2 \oplus C_{2m}) + (2m-1) \ \text{ and } \ \eta(G) > \mathsf{D}(G) + 2m,
\]
however $k_{\mathsf{D}}(G) \ge 2$, 
thus showing that \cite[Remark 5.3.2]{FreezeSchmid10} is nearly optimal in the sense that for $k_{\mathsf{D}}(G)=1$ to hold, the condition $\eta(G) \le \mathsf{D}(G) + \exp(G)$ stated in this remark cannot be replaced by a much weaker inequality.

\section{Auxiliary results}

In this section we establish several auxiliary results. In some cases, we prove results which are slightly more general than what is needed for our immediate purpose, but we mostly focus on the needs of the present paper.

Our first lemma shows that extremal sequences with respect to the Erd\H{o}s--Ginzburg--Ziv constant for groups of rank at most two are stable in the sense that changing a unique element cannot yield another extremal sequence. It could be interesting to consider this problem for more general groups, and to determine for specifc groups the exact number of elements one has to change to get another extremal example.  

\begin{lemma}
\label{stability_s} 
Let $H \simeq C_{m} \oplus C_{mn}$, where $m,n \ge 1$ are integers and $m $ satisfies Property D. 
Let  $S_1,S_2$ be sequences over $H$ of length $\mathsf{s}(H)-1$ not containing any zero-sum subsequence of length $\exp(H)$. 
If $|\gcd(S_1, S_2)| \ge \mathsf{s}(H)-2$, then $S_1 = S_2$. 
\end{lemma}

\begin{proof}
Let $T=\gcd(S_1, S_2)$. By assumption, we have $S_1 = g_1 T$ and $S_2 = g_2 T$ with $g_1, g_2 \in G$.

\medskip
First, assume that $m \ge 3$. We know by Theorem \ref{Wolfgang bis_v} that $\mathsf{v}_g(S_i) \equiv m-1 \pmod{m}$ for each $g \in \supp(S_i)$. 
It thus follows that  $\mathsf{v}_{g_1}(T) \equiv m-2 \pmod{m}$. Since this is non-zero it follows that $g_1 \in \supp(S_2)$, and thus 
$  \mathsf{v}_{g_1}(S_2) \equiv m-1 \pmod{m}$. Yet,  $\mathsf{v}_{g_1}(S_2) = \mathsf{v}_{g_1}(g_2) + \mathsf{v}_{g_1}(T)$. Since $\mathsf{v}_{g_1}(T) \equiv m-2 \pmod{m}$ we must have  $\mathsf{v}_{g_1}(g_2) \neq 0$, that is, $g_1 = g_2$. 
This proves the claim in the case $m \ge 3$. 

\medskip
Now, assume that $m = 2$. 
By Theorem \ref{Wolfgang bis_v}, every sequence $S$ over $H$ containing no zero-sum subsequence of length $2n$ can be decomposed as 
\[
S = c^{2t-1} (b_1+c)^{2(n+1 -t)-1}(b_2+c)^{2s-1} \left(-b_1 + b_2+c \right)^{2(n+1-s)-1},
\]
where $c \in H$, $\{b_1,b_2\}$ is a generating set of $H$ with $\ord(b_2) = 2n$ and $s,t \in [1,n]$, such that either $\{b_1,b_2\}$ is an independent generating set of $H$ or $s=t= n$. 
In particular, one has
\[
\sigma\left(c^{2t-1} (b_1+c)^{2(n+1 -t)-1}\right) = -(2t-1)b_1,
\]
and 
\[
\sigma\left((b_2+c)^{2s-1} \left(-b_1 + b_2+c \right)^{2(n+1-s)-1}\right) = (2s-1)b_1.
\]
Therefore, either $\{b_1,b_2\}$ is an independent generating set of $H$, in which case $\ord(b_1)=2$ (see the comments after Theorem \ref{Wolfgang bis_v} in Section \ref{sec_prel}) so that $\sigma(S)=b_1-b_1=0$, or $s=t=n$, in which case $\sigma(S)=b_1-b_1=0$ also.
It thus follows in both cases that $g_1+\sigma(T) = \sigma(S_1) = 0 = \sigma(S_2) = g_2 + \sigma(T)$, which yields $g_1=g_2$ indeed.

Finally, let us consider the case $m=1$. Then, we know by the results recalled before  Theorem \ref{Wolfgang_v} that $\mathsf{v}_g(S_i) \equiv n-1 \pmod{n}$ for each $g \in \supp(S_i)$. For $n \ge 3$ we can argue as in the case $m \ge 3$. For $n\le 2$, that is for $C_1$ and $C_2$, the claim is trivial as there is only one sequence of length $\mathsf{s}(H)-1$ not containing any zero-sum subsequence of length $\exp(H)$.
\end{proof}

The analogous result for the $\eta$-constant holds true as well; we record it for its own sake. 

\begin{lemma}
\label{stability_eta} 
Let $H \simeq C_{m} \oplus C_{mn}$ where $m,n \ge 1$ are integers. 
Let  $S_1,S_2$ be sequences over $H$ of length $\eta(H)-1$ not containing any short zero-sum subsequence. 
If $|\gcd(S_1, S_2)| \ge \eta(H)-2$, then $S_1 = S_2$. 
\end{lemma}
\begin{proof}
For $m \ge 3$, the same argument as in Lemma \ref{stability_s} works. 
For $m=2$, let $S_1,S_2$ be two sequences over $H$ of length $\eta(H)-1$ not containing any short zero-sum subsequence and such that $|\gcd(S_1, S_2)| \ge \eta(H)-2$.
Since $S_1' = S_10^{\exp(H)-1}$ and $S_2' = S_20^{\exp(H)-1}$ are two sequences over $H$ of length $\mathsf{s}(H)-1$ not containing any zero-sum subsequence of length $\exp(H)$ and such that $|\gcd(S_1, S_2)| \ge \mathsf{s}(H)-2$, Lemma \ref{stability_s} applies and yields $S'_1=S'_2$, that is to say $S_1=S_2$. 
Finally, for $m=1$ and $n \ge 3$ the claim follows from the fact that a sequence 
 over $H$ of length $\eta(H)-1$ not containing any short zero-sum subsequence is of the form $h^{n-1}$ for some generating element of $h \in  H$ (recall the results  before  Theorem \ref{Wolfgang_v}). The case $m=1$ and $n \le 2$ is trivial.
\end{proof}

Next we obtain results on the set of restricted subsums of sequences that are extremal examples with respect to the $\eta$-constant and the Erd\H{o}s--Ginzburg--Ziv constant. The result we obtain is reminiscent of the condition in the definition of the $\nu$-invariant (see, e.g., \cite{Gao_rank3}).

\begin{lemma}
\label{lem_sums}
Let $H \simeq C_{m} \oplus C_{mn}$ where $m,n$ are positive integers and $n \ge 2$.
\begin{enumerate}
\item Let $S$ be a sequence over $H$ of length $|S|= \eta(H)-1$ not containing any short zero-sum subsequence. 
Then $\Sigma_{\le mn - 2} (S) \supseteq H \setminus ((-k' + K)  \cup \{0\} )$ for a proper subgroup $K$ and some $k' \notin K$. 
\item Suppose that $m$ has Property D. Let $S$ be a sequence over $H$ of length $|S|= \mathsf{s}(H)-1$ not containing any zero-sum subsequence of length $\exp(H)$. Then $\Sigma_{mn - 2} (S) \supseteq H \setminus (- k' + K)$ for a proper subgroup $K$ and some $k' \notin K$. 
\end{enumerate}
\end{lemma}
The condition $n \ge 2$ is necessary. Indeed, the claim is not true for groups of the form $C_m^2$. To see this it suffices to note that for $\{b_1, b_2\}$ an independent generating set and  $S= b_1^{m-1}b_2^{m-1} (b_1+b_2)^{m-1}$ the set 
$\Sigma_{\le m - 2} (S)$ contains no element of $-b_1 + \langle b_2 \rangle $ and 
$-b_2 + \langle b_1 \rangle$.

\begin{proof}
We first deal with the main case $m \ge 2$. 

\medskip
(1). Let $S$ be a sequence over $H$ of length $|S|= \eta(H)-1$ not containing any short zero-sum subsequence. By Theorem \ref{Wolfgang_v} we know that 
\[
S = b^{m-1}_1b^{sm-1}_2 \left(-xb_1 + b_2\right)^{(n+1-s)m-1}
\]
where $\{b_1,b_2\}$ is a generating set of $H$ with $\ord(b_2) = mn$, $s \in [1,n]$, $x \in [1, m]$ with $\gcd(x,m)=1$ and either  
$\{b_1,b_2\}$ is an independent generating set of $H$, or $s= n$ and $x=1$.

\medskip
Let $d \in[1,n]$ such that $\ord(b_1)=md$ and, whenever $d > 1$,  let $\ell \in [1,d-1]$ relatively prime to $d$ such that $mb_1=\ell m(n/d)b_2$; 
 as recalled after Theorem  \ref{Wolfgang bis_v} this is always possible. 

\medskip
We now distinguish the following three cases. 

\medskip
\noindent
\textbf{Case 1.} 
$d=1$, that is to say $\{b_1,b_2\}$ is an independent generating set of $H$. 
In particular, $mb_1=0$.
Let $h = a_1 b_1 + a_2 b_2$ with $a_1\in [0,m-1]$ and $a_2 \in [0,mn-1]$.
If $a_1=0$, then $h \notin \Sigma_{\le mn-2}(S)$ only if $a_2=0$ or $a_2=mn-1$, that is to say only if $h=0$ or $h=-b_2$.
If $a_1\neq 0$ and $a_2 \le m - 1$, we observe that $h \in \Sigma_{a_1+a_2}(S)$ and $a_1+a_2 \le 2m -2 \le mn-2$.
Now, assume that $a_1 \neq 0$ and $a_2 \ge m$. 
Let $v \in [1,m-1]$ be the unique integer such that $a_1 \equiv -vx \pmod{m}$.
In particular, one has $v \le m-1 < a_2$ so that $1 \le a_2-v \le mn-2 < mn-1=ms-1+m(n-s)$.
Therefore, there exists $q \in [0,n-s]$ such that $0 \le a_2-v-qm \le ms-1$.
Such an integer $q$ readily satisfies $1 \le v+qm \le m-1+m(n-s)=(n+1-s)m-1$.
As a consequence, 
\[
S'=(-xb_1+b_2)^{v+qm} b^{a_2 - v - qm}_2
\] 
is a subsequence of $S$ of length $|S'|=a_2$ such that
\[
\sigma(S') = -vxb_1 + vb_2 - vb_2 -qx(mb_1) + qmb_2 - qmb_2  + a_2b_2 = a_1b_1 + a_2b_2= h,
\]
so that $h \in \Sigma_{\le a_2}(S)$ and the claim follows with $k'= b_2$ and $K = \langle b_1 \rangle$.

\medskip
\noindent
\textbf{Case 2.} $1 < d < n$, or $d=n$ and $\ell \ge 2$. 
Let $h = a_1 b_1 + a_2 b_2$ with $a_1\in [0,m-1]$ and $a_2 \in [0,mn-1]$.
If $a_1=0$ then $h \notin \Sigma_{\le mn-2}(S)$ only if $a_2=0$ or $a_2=mn-1$, that is to say only if $h=0$ or $h=-b_2$.
Now, assume that $a_1 \neq 0$. 
Then, either $a_1 + a_2 \le mn-2$ in which case $h \in \Sigma_{\le mn-2}(S)$ or $a_2 \ge mn-1 - a_1 \ge mn-1 - (m-1) = m(n-1)$ so that
\[
2mn-1-m \ge mn-1 + mn \frac{d-1}{d} \ge a_2+\ell m \frac{n}{d} \ge m(n-1) + 2m = mn + m.
\]
Therefore, setting $a'_2=a_2+\ell m(n/d)-mn$, we have $a'_2 \in [m,mn-2]$ and since $m-a_1 \in [1,m-1]$ we obtain that
\[
S'=(-b_1+b_2)^{m-a_1} b^{a'_2-(m-a_1)}_2
\] 
is a subsequence of $S$ of length $|S'|=a'_2 \le mn-2$ verifying
\[
\sigma(S')=a_1b_1-mb_1+a'_2b_2=a_1b_1-mb_1 + \ell m \frac{n}{d}b_2 +a_2b_2=h,
\]
so that $h \in \Sigma_{\le mn-2}(S)$ and the claim follows with $k'= b_2$ and $K = \{0\}$.

\medskip
\noindent
\textbf{Case 3.} $d=n$ and $\ell=1$. In this case, $mb_1=mb_2$.
Let $h = a_1 b_1 + a_2 b_2$ with $a_1\in [0,m-1]$ and $a_2 \in [0,mn-1]$.
If $a_1=0$ then $h \notin \Sigma_{\le mn-2}(S)$ only if $a_2=0$ or $a_2=mn-1$, that is to say only if $h=0$ or $h=-b_2$.
Now, assume that $a_1 \neq 0$. 
Then, either $a_1 + a_2 \le mn-2$ in which case $h \in \Sigma_{\le mn-2}(S)$ or $a_1+a_2=mn-1$ or $a_1+a_2 \ge mn$ so that
\[
2mn-2 \ge mn-1+m-1\ge a_1+a_2 \ge mn.
\]
Therefore, setting $a'_2=a_1+a_2-mn$, we have $a'_2 \in [0,mn-2]$ and since $m-a_1 \in [1,m-1]$ we obtain that
\[
S'=(-b_1+b_2)^{m-a_1} b^{a'_2}_2
\] 
is a subsequence of $S$ of length $|S'|=m-a_1+a'_2=m+a_2-mn \le m + mn - 1 - mn \le m-1$ verifying
\[
\sigma(S')=a_1b_1-mb_1+mb_2 -a_1b_2 + a_1b_2 + a_2b_2=h,
\]
so that $h \in \Sigma_{\le mn-2}(S)$ and the claim follows with $k'= b_2$ and $K = \left\langle b_1-b_2 \right\rangle$.

\medskip
(2). Suppose that $m$ has Property D. Let $S$ be a sequence over $H$ of length $|S|= \mathsf{s}(H)-1$ not containing any zero-sum subsequence of length $\exp(H)$. 
By Theorem \ref{Wolfgang bis_v} we know that
\[
S = c^{tm-1} (b_1+c)^{(n+1 -t)m-1}(b_2+c)^{sm-1} (-xb_1 + b_2+c)^{(n+1-s)m-1}
\]
where $c\in H$, $\{b_1,b_2\}$ is a generating set of $H$ with $\ord(b_2) = mn$, $s,t \in [1,n]$, $x \in [1, m]$ with $\gcd(x,m)=1$ and either $\{b_1,b_2\}$ is an independent generating set of $H$, or $s=t= n$ and $x=1$.
 
\medskip
Since $\Sigma_{mn - 2} (-c + S)  = 2c + \Sigma_{mn - 2} (S)$, we can assume without loss of generality that $c= 0$. 
We now distinguish the following two cases. 
 
\medskip
\noindent
\textbf{Case 1.} 
If $\{b_1,b_2\}$ is an independent generating set of $H$, we see that $S = 0^{tm-1} b_1^{(n -t)m}T$, where $T$ is a sequence of length $\eta(H)-1$ that has no short zero-sum subsequence. 
By point (1), it follows that then $\Sigma_{\le mn - 2} (T) \supseteq H \setminus ((-k' + K)  \cup \{0\} )$ for a proper subgroup $K$ and some $k' \notin K$.
Therefore, it suffices to assert that $\{0\} \cup \Sigma_{\le mn - 2} (T) \subseteq \Sigma_{mn - 2} (0^{tm-1} b_1^{(n -t)m}T)$.
Indeed, note that for each subsequence $T' \mid T$ of length at most $mn-2=mt-1+m(n-t)-1$, and any integer $t' \in [0,n-t]$ such that
\[
mn-2 \ge |T'| +mt' \ge m(n-t)-1,
\]
we obtain
\[
0 \le mn-2-|T'|-mt' \le mt-1
\]
so that, since $mb_1=0$ in this case, the sequence $0^{mn-2 - |T'| - mt'} b_1^{mt'}T'$ is a subsequence of $S$ of length $mn-2$ with the same sum.
The fact that we get $0$ in addition to  $\Sigma_{\le mn - 2} (T)$ is due to the fact that $T'$ can be chosen to be the empty sequence.
 
\medskip
\noindent
\textbf{Case 2.} 
If $s=t= n$ and $x=1$, we see that $S = 0^{mn-1}T$, where $T$ is a sequence of length $\eta(H)-1$ that has no short zero-sum subsequence. 
By point (1), it follows that then $\Sigma_{\le mn - 2} (T) \supseteq H \setminus ((-k' + K)  \cup \{0\} )$ for a proper subgroup $K$ and some $k' \notin K$. 
We assert that $\{0\} \cup \Sigma_{\le mn - 2} (T) \subseteq \Sigma_{mn - 2} (0^{mn-1}T)$, then the claim is proved.  
As above, it suffices to note that for each subsequence $T'$ of $T$, the sequence $0^{mn-2 - |T'|}T'$ is a subsequence of $S$ of length $mn-2$ with the same sum. 
The fact that we get $0$ in addition to  $\Sigma_{\le mn - 2} (T)$ is due to the fact that $T'$ can be chosen to be the empty sequence. 

\medskip
To finish the argument  we consider the case $m=1$. 
For assertion (1), we have a sequence $S$ over $H$ of length $|S|= \eta(H)-1=n-1$ not containing any short zero-sum subsequence. By the results recalled before  Theorem \ref{Wolfgang_v}  we know that $S= b^{n-1}$ for some generating element $b$ of $H$. 
It follows that $\Sigma_{\le n - 2} (S) = \{b, 2b, \dots, (n-2)b\} $. Thus the claim is established with $K = \{0\}$.   
For assertion (2), we have a sequence $S$ over $H$ of length $|S|= \mathsf{s}(H)-1=2n-2$ not containing any zero-sum subsequence of length $n$. By the results recalled above we know that $S= c^{n-1} (c+b)^{n-1}$ for some generating element $b$ of $H$. 
Without loss of generality we can assume that $c=0$. It follows that $\Sigma_{ n - 2} (S) = \{0,b, 2b, \dots, (n-2)b\} $. Thus the claim is established with $K = \{0\}$.   
\end{proof}

The following lemma slightly develops a well-known technique useful to establish  Conjecture \ref{GaoConjecture}; see, e.g., \cite[Proposition 2.7]{Gao03} or \cite[Proposition 2.3']{SuryThanga02} for earlier versions. We do not need the second part in this paper, but include it as it might be useful elsewhere.  We note that the condition in the lemma is trivial for $\exp(G)\le 4$.

\begin{lemma}
\label{lem_exp-1}
Let $G$ be a finite abelian group. 
The following two statements hold.
\begin{enumerate}
\item Let $S$ be a sequence over $G$ of length $\eta(G) +  \exp(G)-1$. 
Let $C \mid S$ be a subsequence such that there exists some $h \in G$ with $j h \in \Sigma_{j}(C)$ for each $j \le |C|$. 
If $|C|  \ge \lfloor (\exp(G ) - 1)/2 \rfloor$, then $S$ has a zero-sum subsequence of length $\exp(G)$.
\item Let $S$ be a sequence over $G$ of length $(\eta(G)-1) +  \exp(G)-1$ that does not contain any zero-sum subsequence of length $\exp(G)$. 
Let $C \mid S$ be a subsequence such that there exists some $h \in G$ with $j h \in \Sigma_{j}(C)$ for each $j \le |C|$.
If $|C|  \ge \lfloor (\exp(G ) - 1)/2 \rfloor$, then $-h+S$ has a  subsequence of length $\eta(G)-1$ without any short zero-sum subsequence.
\end{enumerate}
\end{lemma}

\begin{proof}
Without loss of generality suppose that $h=0$.

\medskip
\noindent
(1).  Consider $ SC^{-1}$. Let $T \mid SC^{-1}$ be a short zero-sum subsequence (or the empty sequence) of maximal length. 
If $|T| > \exp(G)/2$, then $|C| \ge \exp(G) - |T|$. 
Thus, $C$ has a subsequence $C'$ of length $\exp(G) - |T|$ with sum $0$.
Since $|TC'| = \exp(G)$ and its sum is $0$, the argument is complete in this case. 
Consequently, we can assume that  $|T| \le \exp(G)/2$. 
It follows that  $SC^{-1}T^{-1}$ has no short zero-sum subsequence. 
To see this it suffices to note that a short zero-sum subsequence $T'$ would satisfy $|T'| \le |T| \le \exp(G)/2$ and thus $TT'$ would also be a short zero-sum subsequence of $SC^{-1}$ contradicting the maximality of $T$. 
If $|C| \ge \exp(G) - |T|$, then we get a zero-sum subsequence of length $\exp(G)$ as above. 
Thus, $|C| + |T| \le \exp(G) - 1$ and thus  $|SC^{-1}T^{-1}| \ge \eta(G)$, contradicting the fact that $SC^{-1}T^{-1}$ has no short zero-sum subsequence.

\medskip
\noindent
(2). The proof is similar to the first part. Consider $ SC^{-1}$. 
Let $T \mid SC^{-1}$ be a short zero-sum subsequence (or the empty sequence) of maximal length. 
If  $|T| > \exp(G)/2$ then, as in (1), $C$ has a subsequence $C'$ of length $\exp(G) - |T|$ with sum $0$, yielding again a zero-sum subsequence of $S$ of length $\exp(G)$, which is a contradiction. 
Consequently, we have  $|T| \le \exp(G)/2$. It follows that, as in (1), $SC^{-1}T^{-1}$ has no short zero-sum subsequence. 
If $|C| \ge \exp(G) - |T|$, then we get a contradiction as above. Thus, $|C| + |T| \le \exp(G) - 1$ and thus  $|SC^{-1}T^{-1}| \ge \eta(G)-1$, establishing our claim.
\end{proof}

\section{Proofs of the main results}

In this section we prove our Theorems \ref{eta}, \ref{EGZ} and \ref{D_k}. 
The proofs of the latter two will rely on the first.

\begin{proof}[Proof of Theorem \ref{eta}]
We start by discussing the lower bounds.  If $n=1$, then \cite[Proposition 3.1(3)]{EdelGero07} yields
\[
\eta(C_2 \oplus C_{2m} \oplus C_{2m}) \ge 2+ (2^2-1)(2m - 1 + 2 - 1) = 6m + 2,
\]
and if $n \ge 2$, then \cite[Lemma 3.2]{EdelGero07} gives
\[
\eta(C_2 \oplus C_{2m} \oplus C_{2mn}) \ge 2(\mathsf{D}(C_2 \oplus C_{2m})-1)+2mn = 4m + 2mn.
\]

We can now turn to the upper bounds. 
Let $H$ be a subgroup of $G$ isomorphic to $C_m \oplus C_{mn}$ such that $G \slash H$ is isomorphic to $C_2^3$. 
We apply the inductive method with 
\[
H \hookrightarrow G \overset{\pi}{\to} G \slash H.
\]

First, suppose $n=1$, that is, let $G = C_2 \oplus C_{2m} \oplus C_{2m}$.  
By \cite[Proposition 5.7.11]{GeroKoch06} and since $\eta(C_2^3)= 8$ and $\eta(C_m^2)= 3m-2$, it follows that 
\[
\eta(G) \le  (\eta(H) - 1 ) \exp( G \slash H )  + \eta( G \slash H ) = ( 3 m - 3 ) 2 + 8 = 6m + 2.
\] 

Second, suppose $n\ge 2$. 
Let $S$ be a sequence over $G$ with $|S|=4m + 2mn$. 
We have to show that $S$ has a short zero-sum subsequence. 
Note that applying \cite[Proposition 5.7.11]{GeroKoch06} yields only an upper bound of $2mn +4m +2$ 
and a more refined analysis is needed. 

Since $4m + 2mn = 2(2m + mn - 4) + \eta(C_2^3)$, it follows that 
there exist $(2m +mn - 3)$ non-empty and disjoint subsequences of $S$, say, $S_1 \cdots S_{2m +mn - 3} \mid S$, 
with $|S_i| \le 2$ and $\sigma(\pi(S_i))=0$ for each $i$. 
Let $T$ be the subsequence of $S$ such that  $S_1 \cdots S_{2m +mn - 3} T= S$. 
We note that $|T| \ge 6$.

\medskip
We observe that $R= \sigma(S_1)\cdots \sigma(S_{2m + mn -3})$ is a sequence over $H$ of length $\eta(H)-1$. 
If $R$ has a non-empty zero-sum subsequence of length at most $mn$, that is one that is short relative to $H$,  we can complete the argument as follows. 
We note that if $\sum_{i \in I} \sigma(S_i)=0$ for some $\emptyset \neq I \subseteq [1, 2m +mn -3]$ with at most $mn$ elements, then $\prod_{i\in I} S_i$ is a non-empty zero-sum subsequence of $S$ of length at most  $2|I| \le 2mn$. Thus we assume $R$ does not have a short zero-sum subsequence. This means that $R$ fulfills the conditions of Lemma \ref{lem_sums}. Thus, we get that there exist a proper subgroup $K$ of $H$ and some $k' \in H\setminus K$ such that the complement (in $H$) of $\Sigma_{\le mn -2} (R)\cup \{0\}$ is contained in $-k' + K$.

\medskip
We continue by analyzing the sequence $T$. First, we note that we may assume that $\pi(T)$ does not have a non-empty zero-sum subsequence of length at most $2$. Otherwise, let $S_0\mid T$ with $1 \le |S_0| \le 2$ and we consider the sequence $\sigma(S_0)R$ over $H$ that has length $\eta(H)$. Thus it has a short (relative to $H$) zero-sum subsequence. Using the same argument as above, this yields a short (relative to $G$) zero-sum subsequence of $S$.

\medskip
Second, somewhat in the same vein, we note that for each $S_0\mid T$ with $1 \le |S_0| \le 4$ such that $\pi(S_0)$ is a zero-sum sequence we may assume that $\sigma(S_0) \notin - (\Sigma_{\le mn -2} (R)\cup \{0\})$. 
To see this, just observe that otherwise we would get a non-empty zero-sum subsequence of $\sigma(S_0)R$ of length at most $mn-1$, which contains $\sigma(S_0)$.
This then establishes the existence of a non-empty zero-sum subsequence of $S$ of length at most $|S_0|+ 2(mn-2) \le 2mn$, that is to say a short zero-sum subsequence of $S$.  
Therefore, recalling what we know about  $\Sigma_{\le mn -2} (R)\cup \{0\}$, we get a short zero-sum subsequence of $S$ unless $\sigma(S_0) \in    k' + K$ for each $S_0 \mid T$ with $1 \le |S_0| \le 4$ and $\sigma(\pi(S_0))=0$. 

\medskip
It remains to show that  $\sigma(S_0) \in    k' + K$ for each $S_0 \mid T$ with $1 \le |S_0| \le 4$ and $\sigma(\pi(S_0))=0$ is impossible. We know that $|T| \ge 6$ and that $\pi(T)$ consists of distinct non-zero elements, as otherwise $\pi(T)$ would contain a non-empty zero-sum subsequence of length at most $2$, which we excluded above. 
Fixing an appropriate independent generating set $\{e_1, e_2, e_3 \}$ of $G/H$ we may assume that $\supp(\pi(T))$ contains all non-zero elements except $e_1 + e_2 + e_3$. For $I \subseteq \{1,2,3\}$ with two elements, let $e_I=\sum_{i\in I}e_i$. We note that $\pi(T)$ has  at least the following zero-sum subsequences of length at most $4$: $V_k=e_ie_je_{\{i,j\}}$, 
$V_0=e_{\{1,2\}}e_{\{2,3\}}e_{\{1,3\}}$, and $V'_{i}=e_{\{i,j\}}e_{\{i,k\}}e_je_k$ for $\{i,j,k\}=\{1,2,3\}$. Let $T_i$ (respectively, $T_i^{\prime}$) denote the subsequence of $T$ whose image under $\pi$ is $V_i$ (respectively, $V_i^{\prime}$).
We want to show that at least one of these sequences $T_i^{(\prime)}$  has a sum that is not in $k' +K$. Assume to the contrary that the sum of each of these sequences is in $k'+K$.  Now, note that $V_0V_1V_2V_3=V_1'V_2'V_3'$ and thus $T_0T_1T_2T_3=T_1'T_2'T_3'$. However, that yields $\sigma(T_0T_1T_2T_3) \in 4k' + K$ while $\sigma(T_1'T_2'T_3') \in 3k' +K$. Since $k' \notin K$, this is a contradiction. Thus, $\sigma(S_0) \in    k' + K$ for each $S_0 \mid T$ with $1 \le |S_0| \le 4$ and $\sigma(\pi(S_0))=0$ is indeed impossible, and consequently $S$ has a short zero-sum subsequence. 
\end{proof}

\medskip
We continue with the proof of our result on the Erd{\H o}s--Ginzburg--Ziv constant. 

\begin{proof}[Proof of Theorem \ref{EGZ}]
Since $\mathsf{s}(G) \ge \eta(G) + \exp(G) -  1$ for every finite abelian group $G$ (see the remark before Conjecture \ref{GaoConjecture}), Theorem \ref{eta} readily yields the desired lower bounds. We now show that these bounds are indeed optimal. Our strategy is to obtain a situation in which we can invoke Lemma \ref{lem_exp-1} and then apply Theorem \ref{eta}. As already noted, we can always apply Lemma \ref{lem_exp-1} if $\exp(G) \le 4$. Thus, we assume that $mn >2$.  

\medskip
Let $H$ be a subgroup of $G \simeq C_2 \oplus C_{2m} \oplus C_{2mn}$ isomorphic to $C_m \oplus C_{mn}$ such that $G \slash H$ is isomorphic to $C_2^3$. 
As before, we apply the inductive method with 
\[
H \hookrightarrow G \overset{\pi}{\to} G \slash H.
\]

For $n=1$, by \cite[Proposition 5.7.11]{GeroKoch06} and since $\mathsf{s}(C_2^3)=9$ and $\mathsf{s}(C_m^2) = 4m-3$, it follows that 
\[
\mathsf{s}(G) \le  (\mathsf{s}(H) - 1 ) \exp(G \slash H)  + \mathsf{s}(G \slash H) = (4m - 4 ) 2 + 9 = 8m + 1.
\]

Now, suppose $n\ge 2$ and $m$ has Property D. 
Let $S$  be a sequence over $G$ with $|S|=4m + 4mn - 1$. Assume for a contradiction that  $S$ has no zero-sum subsequence of length $\exp(G)=2mn$.

 \medskip 
Since $4m + 4mn - 1 = 2(2m + 2mn - 5) + \mathsf{s}(C_2^3)$, it follows that there exist $(2m +2mn - 4)$ disjoint subsequences of $S$,  say, $S_1 \cdots S_{2m +2mn - 4} \mid S$ with $|S_i| = 2$ and $\sigma(\pi(S_i))=0$ for each $i$. Let $T$ be the subsequence of $S$ such that  $S_1 \cdots S_{2m +2mn - 4} T= S$. We note that $|T| = 7$.

\medskip
We observe that $R= \sigma(S_1)\cdots \sigma(S_{2m + 2mn -4})$ is a sequence over $H$ of length $\mathsf{s}(H)-1$. If $R$ has a  zero-sum subsequence of length $mn$ we are done, since  $\sum_{i \in I} \sigma(S_i)=0$ for some $I \subseteq [1, 2m + mn - 4]$ with $|I| =  mn$ implies that $\prod_{i\in I} S_i$ is a zero-sum subsequence of $S$ of length  $2|I| = 2mn$. Thus, the assumption that $S$ has no zero-sum subsequence of length $2mn$, implies that  $R$ does not have  a zero-sum subsequence of length $mn$. Hence $R$ fulfills the conditions of Lemma \ref{lem_sums}.

In addition, we note that if $\pi(T)$ still has a zero-sum subsequence of length $2$, then we also get a zero-sum subsequence of $S$ of length $2mn$. Thus, we get that $\pi(T)$  has no zero-sum subsequence of length $2$.

\medskip
We continue by establishing an auxiliary fact. 
\begin{itemize}
\item[\textbf{F}.] If $g \mid T$ and $h \mid S$ such that $\pi(g) = \pi(h)$, then $h=g$ or $S$ contains a zero-sum subsequence of length $\exp(G)$.
\end{itemize}

Assume there are distinct $g,h$ with $g \mid T$ and $h \mid S$ such that $\pi(g) = \pi(h)$. Since $\pi(T)$ does not contain a zero-sum subsequence of length $2$, it follows that $h \nmid T$ and thus $h \mid S_i$ for some $i$, say $i=1$. Then $R'=\sigma(gS_1h^{-1}) \sigma(S_2) \cdots \sigma(S_{2m + 2mn -4})$ is a sequence over $H$ of length $\mathsf{s}(H)-1$. If this sequence contains a zero-sum subsequence of length $mn$, then as above $S$ contains a zero-sum subsequence of length $2mn$. 
Thus, we know that it does not contain such a subsequence, and since we have $|\gcd(R,R')| \ge \mathsf{s}(H)-2$, Lemma \ref{stability_s} gives that $R=R'$. This means that $\sigma(gS_1h^{-1})= \sigma(S_1)$, contradicting the assumption that $g,h$ are distinct. This establishes \textbf{F}.

\medskip 
After these preparations, we proceed to show that the conditions of Lemma \ref{lem_exp-1} are satisfied. For $r \in \supp(R)$,  let  $I_r \subseteq [1, 2m + mn - 4] $ denote the set of all $i$ such that $\sigma(S_i)=r$ and let $Q_r = \prod_{i \in I_r} S_i$.  
If for some $i \in I_r$ we have that $S_i$ does not contain two distinct elements, say $S_i = h_i^2$ for some $h_i \in G$,  then it is not hard to see that $jh_i \in \Sigma_{j}(Q_r)$ for every $j \in [1,|Q_r|]$; just note that $\sigma(S_j)= r= 2h_i$ for every $j \in I_r$.  

Let $v_r =  \mathsf{v}_r(R)$. We have $|Q_r|= 2v_r$. If $v_r \ge (n+1)m/2 -1$, then we have $|Q_r| \ge mn + m-2 \ge mn -1  =  \lfloor (\exp(G)-1)/2\rfloor$. 
Thus, if for such an $r$ there is some $i \in I_r$ with $S_i = h^2_i$, then  by Lemma \ref{lem_exp-1} we would get a zero-sum subsequence of $S$ of length $\exp(G)=2mn$.

\medskip
Since $\pi(S_i)$ is a zero-sum sequence of length $2$ over $G \slash H \simeq C_2^3$
we have $\pi(S_i) = e^2$ for some $e \in G\slash H$. If we have $\pi(S_i) = e^2$ for some $e \mid \pi(T)$, then by \textbf{F} we know that $S_i= h_i^2$; this is because both elements of $S_i$ are equal to the one corresponding element in $T$. 

\medskip
Thus, the only situation in which we cannot establish by \textbf{F} that there is some $i \in I_r$ with $S_i = h_i^2$ is that $\pi(Q_r) = e_0^{|Q_r|}$ where $e_0$ is the \emph{unique} element of $G/H$ not in $\pi(T)$. 
 
\medskip
Let $r,r' \in \supp(R)$ be the two elements with the greatest multiplicity in $R$. We know that their respective multiplicities $v_r,v_{r'}$ are at least $(n+1)m/2 -1$; and if $m=1$ then $v_r = v_{r'}=n-1$ (see the results recalled before Theorem \ref{Wolfgang_v}).  Suppose $\pi(Q_{r})  = e_0^{2v_{r}}$ where $e_0$ is the unique element of $G/H$ not in $\pi(T)$, and likewise for $r'$. It suffices to show that there exists some $i \in I_{r} \cup I_{r'}$ such that $S_i =h_i^2$ for some $h_i \in G$. 

\medskip
We proceed to show this is always the case. Since $mn>2$, we get that $v_{r}, v_{r'} \ge 2$.
Let $i \in  I_{r}$ and $i' \in  I_{r'}$. Say $S_{i} = s_1s_2$ and $S_{i'} = s_1's_2'$. 
Assume for a contradiction that each of these two sequences consists of two distinct elements. 
We have $\sigma(S_{i}) = s_1  + s_2 = r$ as well as $\sigma(S_{i'}) = s_1'  + s_2' = r'$. 

\medskip
We can consider instead of  $S_{i} = s_1s_2$ and $S_{i'} = s_1's_2'$, the sequences $S_{i}' = s_1s_1'$, $S_{i'}' = s_2s_2'$. 
We set $S_j' = S_j$ for all $j \notin \{i,i'\}$. 
Then $R'= \sigma(S_1')\cdots \sigma(S_{2m + 2mn -4}')$ is a sequence over $H$ of length $\mathsf{s}(H)-1$ and it has no zero-sum subsequence of length $\exp(H)$. 
Since $\supp(R) \subseteq \supp(R')$, the supports are in fact equal. If $m=1$, it is immediate by the results recalled before Theorem \ref{Wolfgang_v} that $R=R'$.  If $m > 1$, then since the multiplicity of each element in $R$ is $m-1 \pmod{m}$ and the same must be true for $R'$, it follows again that in fact $R=R'$. This means that $\{\sigma(S_{i}'),\sigma(S_{i'}')\} = \{r, r'\}$.  

\medskip
Likewise, we can consider instead of  $S_{i} = s_1s_2$ and $S_{i'} = s_1's_2'$, the sequences $S_{i}'' = s_1s_2'$, $S_{i'}'' = s_2s_1'$. 
And we set $S_j'' = S_j$ for all $j \notin \{i,i'\}$. Again, it follows that  $R=R''$.  
This means that $\{\sigma(S_{i}''),\sigma(S_{i'}'')\} = \{r, r'\}$.  

\medskip
Since $\sigma (S_{i}') \neq \sigma (S_{i}'')$, recall that we assumed $s_{1}'\neq s_{2}'$, we get $\{\sigma (S_{i}'), \sigma (S_{i}'')\} =\{r, r'\}$. Without loss of generality we can assume $\sigma (S_{i}')=r$. Then, we get $\sigma (S_{i}'')=r'$ and thus $\sigma (S_{i'}'')= r$. Since also $\sigma (S_{i}) = r$, we obtain $s_{1} + s_{2} = s_{1} + s_{1}' = s_{2} + s_{1}'$ so that $s_{2} = s_{1}'$ and $s_{1} = s_{1}'$ and finally $s_{1} = s_{2}$, which is a contradiction.
\end{proof}

We end with the proof of our result on the multiwise Davenport constants.

\begin{proof}[Proof of Theorem \ref{D_k}] 
When $n \ge 2$, it follows from Theorems \ref{ranktwo}, \ref{D} and \ref{eta} that
\[
\mathsf{D}(G) = \mathsf{D}(C_2 \oplus C_{2m}) + (2mn-1) \ \text{ and } \ \eta(G) \le \mathsf{D}(G) + 2mn,
\]
so that \cite[Theorem 6.1.5(1)]{GeroKoch06} (see also \cite{DelOrdQui01} and \cite{Koch92}) yields the desired result.

\medskip
Now, assume that $n=1$. If $k=1$, Theorem \ref{D} readily gives
\[
\mathsf{D}_1(G) = \mathsf{D}(G) = 2m + 2mn.
\]
If $k \ge 2$, let $H$ be a subgroup of $G$ isomorphic to $C_m^2$ such that $G \slash H$ is isomorphic to $C_2^3$. 
On the one hand, \cite[Lemma 3.7]{DelOrdQui01} gives $\mathsf{D}_0(C^3_2)=3$, $k_{\mathsf{D}}(C^3_2)=2$. 
By Theorem \ref{ranktwo},
\[
\mathsf{D}_k(C_m^2) = m + km -1 \ge 2 = k_{\mathsf{D}}(C^3_2).
\] 
Therefore, we have (the first inequality following from \cite[Proposition 2.6]{DelOrdQui01}, which can be proved using the inductive method)

\[\begin{aligned}
\mathsf{D}_k(C_2\oplus C_{2m} \oplus C_{2m}) & \le \mathsf{D}_{\mathsf{D}_k(H)}(G/H) \\
                                                                             & =  \mathsf{D}_{\mathsf{D}_k(C_m^ 2)}(C^3_2) \\
                                                                             &  =   \mathsf{D}_{(m(k+1)-1)}(C^3_2) \\
                                                                             &  =   \mathsf{D}_0(C^3_2) + 2(m +km-1) \\
                                                                             &  =   3 + 2(m+km-1) \\
                                                                             &  =   (2m+1) + k(2m).
\end{aligned}\]
On the other hand, let $\{e_1,e_2,e_3\}$ be an independent generating set of $G$ such that $\text{ord}(e_1)=2$ and $\text{ord}(e_2)=\text{ord}(e_3)=2m$.
For $k \ge 2$, we consider the sequence
\[
S_k = (e_2+e_3)^{2(k-1)m-1} e_2^{2m-1} e_3^{2(m-1)} e_1 (e_1+e_2+e_3) (e_1+e_2) (e_1+e_3).
\]
We show that $S_k$ does not contain $k$ disjoint non-empty zero-sum subsequences. 
This yields  $ \mathsf{D}_k(C_2\oplus C_{2m} \oplus C_{2m})  > |S_k| = 2m + k(2m)$.  
We  set 
\[
T =  e_2^{2m-1} e_3^{2(m-1)} e_1 (e_1+e_2+e_3) (e_1+e_2) (e_1+e_3) \mid S_k,
\]
and
\[
U = e_1 (e_1+e_2+e_3) (e_1+e_2) (e_1+e_3) \mid T.
\]

\medskip
Now, suppose we have $A_1 \cdots A_k \mid S_k$ where $A_1,\dots,A_k$ are $k$ non-empty zero-sum subsequences.  Since every non-empty zero-sum sequence is a product of minimal zero-sum subsequences, we can assume that $A_i$ is minimal for every $i \in [1,k]$. 

\medskip
For every $i \in [1,k]$, we set $k_i=\mathsf{v}_{e_2+e_3}(A_i) \in [0,2m]$ and $A'_i=A_i((e_2+e_3)^{k_i})^{-1} \mid T$. 
If $A'_i$ does not contain any element of $U$ 
then $A'_i=e^{2m-k_i}_2e^{2m-k_i}_3$ so that $|A_i|=4m-k_i$.
If $A'_i$ contains at least one of the elements of $U$
then it contains exactly two or four of them. 
Therefore, we obtain the following seven cases:
\begin{itemize}
\item \(A'_i=e_1 (e_1+e_2) e^{2m-(k_i+1)}_2e^{2m-k_i}_3\),  so that \( |A_i|=4m-k_i+1\).
\item  \(A'_i=e_1 (e_1+e_3) e^{2m-k_i}_2e^{2m-(k_i+1)}_3\),  so that  \(|A_i|=4m-k_i+1\).
\item \(A'_i=e_1 (e_1+e_2+e_3) e^{2m-(k_i+1)}_2e^{2m-(k_i+1)}_3\) ,  so that \( |A_i|=4m-k_i\).
\item \(A'_i=(e_1+e_2) (e_1+e_3) e^{2m-(k_i+1)}_2e^{2m-(k_i+1)}_3\),  so that  \(|A_i|=4m-k_i\).
\item \(A'_i=e_1 (e_1+e_2+e_3) (e_1+e_2) (e_1+e_3) e^{2m-(k_i+2)}_2e^{2m-(k_i+2)}_3\),  so that \( |A_i|=4m-k_i\).
\item \(A'_i=(e_1+e_2) (e_1+e_2+e_3) e^{2m-(k_i+2)}_2e^{2m-(k_i+1)}_3\),  so that \( |A_i|=4m-k_i-1\).
\item \(A'_i=(e_1+e_3) (e_1+e_2+e_3) e^{2m-(k_i+2)}_2e^{2m-(k_i+1)}_3\),  so that \( |A_i|=4m-k_i-1\).
\end{itemize}

Note that since $\mathsf{v}_{e_1+e_2+e_3}(S_k)=1$ and $A_1,\dots,A_k$ are disjoint, there is at most one $i \in [1,k]$ such that $|A_i|=4m-k_i-1$.
This yields
\[\begin{aligned}
\displaystyle\sum^k_{i=1} |A_i| & \ge  \displaystyle\sum^k_{i=1}(4m-k_i) -1 \\
                                                  &  =   2k(2m) - \displaystyle\sum^k_{i=1} k_i - 1 \\
                                                  & \ge  2k(2m) - \mathsf{v}_{e_2+e_3}(S_k) - 1 \\
                                                  & =  2k(2m) - (2(k-1)m-1) - 1 \\
                                                  &  =   2m+k(2m) \\
                                                  &  =   |S_k|. 
\end{aligned}\]
Therefore, we obtain $A_1 \cdots A_k = S_k$. Yet, then  $0=\sigma(A_1 \cdots A_k)= \sigma(S_k)= -e_3$, a contradiction.
\end{proof}

\section*{Acknowledgment}

The authors are grateful to the anonymous referee for detailed remarks which allowed for better accuracy.

\end{document}